\documentclass{birkmult}
\usepackage{amsfonts,amssymb,amscd,amsthm}
\input xy
\xyoption{all}
\usepackage{cite}
\usepackage{graphicx}
\usepackage{pifont}
\usepackage{pgf,pgfarrows,pgfnodes,pgfautomata,pgfheaps,pgfshade}

\def\ind{\operatorname{ind}}

\newcommand{\Vect}{\mathop{\rm Vect}}

\def\End{\operatorname{End}}
\def\Td{\operatorname{Td}}
\def\td{\mathrm{Td}}
\def\ch{\operatorname{ch}}

\def\om{\omega}
\def\Ga{G}
\def\Om{\Omega}

\DeclareMathOperator\tr{tr}

\def\lra{\longrightarrow}

\newtheorem{theorem}{Theorem}

\newtheorem{proposition}{Proposition}
\newtheorem{corollary}{Corollary}

\theoremstyle{remark}
\newtheorem{remark}{Remark}

\newtheorem{definition}{Definition}
\newtheorem{assumption}{Assumption}

 \theoremstyle{definition}
 
 \theoremstyle{remark}

 \numberwithin{equation}{section}

\begin{document}
%
%
%
%
%
%
%
%
%
\title[Elliptic theory for operators associated with diffeomorphisms]
 {Elliptic theory for operators associated with \\
diffeomorphisms of smooth manifolds}
\author[Anton Savin]{Anton Savin}

\address{%
Peoples' Friendship University of Russia\\
 Miklukho-Maklaya str. 6\\
117198 Moscow\\
Russia}

\address{%
Leibniz University of Hannover, Institut fur Analysis\\
 Welfengarten 1\\
30167 Hannover\\
Germany}

\email{antonsavin@mail.ru}

\thanks{The work was partially supported by RFBR grant NN~12-01-00577.}
\author[Boris Sternin]{Boris Sternin}
\address{%
Peoples' Friendship University of Russia\\
Miklukho-Maklaya str. 6\\
117198 Moscow\\
Russia}
\address{%
Leibniz University of Hannover, Institut fur Analysis\\
Welfengarten 1\\
30167 Hannover\\
Germany}
\email{sternin@mail.ru}


\date{January 31, 2012}

\begin{abstract}
In this paper we give a survey of elliptic theory for operators associated with
diffeomorphisms of smooth manifolds. Such operators appear naturally in
analysis, geometry and mathematical physics. We survey classical results as
well as results obtained recently. The paper consists of an introduction and
three sections. In the introduction we give a general overview of the 
area of research.  For the reader's convenience here we tried to keep special
terminology to a minimum. In the remaining sections we give detailed
formulations of the most  important results mentioned in the introduction.
\end{abstract}

\maketitle

\section*{Introduction}
\addcontentsline{toc}{section}{Introduction}

The aim of this paper is to give a survey of   index theory for  elliptic
operators associated with diffeomorphisms of smooth manifolds. Recall that the
construction of index theory includes the following main stages:

 \begin{enumerate}
    \item[1)] (finiteness theorem) Here one has to give conditions,
    called ellipticity conditions, under which the operators under consideration are Fredholm in relevant function spaces;
    \item[2)] (index theorem) Here one presents and proves an index formula, that is, an
    expression for the index of an elliptic operator in terms of topological invariants
    of the symbol of the operator and the manifold, on which the operator is defined.
\end{enumerate}

\vspace{1mm}

The first index theorem on higher-dimensional manifolds was the celebrated
Atiyah--Singer  theorem \cite{AtSi0} on the index of elliptic
pseudodifferential operators   ($\psi$DO) on a closed smooth manifold. This
theorem appeared as an answer to a question posed by Gelfand~\cite{Gel1}. Note
that the statement and the proof of the index formula relied on
most up to date methods of analysis and topology and stimulated interactions
between them.

After that index theorems were obtained for many other classes of operators.
In this paper we consider the class of operators associated with
diffeomorphisms of closed smooth manifolds. One of advantages of this theory
is that, besides the mentioned interaction of analysis and topology, here an
important role is played by the theory of dynamical systems.

\vspace{2mm}

\noindent\textbf{1. Elliptic operators for a discrete group of diffeomorphisms
(analytical aspects).} The theory of elliptic operators associated with
diffeomorphisms and the corresponding theory of boundary value problems with
nonlocal boundary conditions go back to the paper by T.~Carleman \cite{Carl1},
where he considered the problem of finding a holomorphic function in a
bounded domain   $\Omega$, which satisfies a nonlocal boundary condition, which
relates the values of the function at a point   $x\in \partial\Omega$ of the
boundary and at the point   $g(x)\in\partial \Omega$, where $g:\partial
\Omega\to \partial \Omega$ is a smooth mapping of period two: $g^2=Id.$ A
reduction of this boundary value problem to the boundary does not  give 
usual integral equation as it war the case with the local boundary condition. Rather, it gives an
integro-functional equation, which we call  \emph{equation associated with
diffeomorphism   $g$}. This paper motivated the study of a more general class of
operators on closed smooth manifolds. Let us give the general definition of
such operators.

On a closed smooth manifold   $M$ we consider operators of the form
\begin{equation}\label{eq-ooper1}
    D=\sum_{g\in G} D_g T_g:C^\infty(M)\to C^\infty(M),
\end{equation}
where:
\begin{itemize}
\item  $G$ is a discrete group of diffeomorphisms of $M$;
\item $
 (T_gu)(x)=u(g^{-1}(x))
$ is the shift operator corresponding to the diffeomorphism $g$; \item
$\{D_g\}$ is a collection of pseudodifferential operators of order $\le m$;
\item $C^\infty(M)$ is the space of smooth functions on   $M$. Of course,  one
can also consider operators acting in sections of vector bundles.
\end{itemize}
Operators \eqref{eq-ooper1}  will be called {\em  $G$-pseudodifferential
operators} ($G$-$\psi$DO) or simply $G$-operators.\footnote{In the literature
such operators are also called functional-differential, nonlocal,
noncommutative operators and operators with shifts.} Such operators were intensively studied
(see the fundamental works of  Antonevich~\cite{Ant2,Ant1}, and also the papers
\cite{AnLe1,AnLe3} and the references cited there). In particular, and extremely important notion of
 {\em symbol of $G$-operator} was introduced there.  More
precisely,  two definitions of the symbol of a $G$-operator were given.  First, the
symbol was defined as a function on the cotangent bundle   $T^*M$ of the
manifold taking values in operators acting on the space   $l^2(G)$ of
square integrable functions on the group. Second, the symbol was defined as
an element of the crossed product~\cite{Zell1}  of the algebra of continuous
functions on the cosphere bundle    $S^*M$ of the manifold and the group $G$.
Further, we introduce the ellipticity condition in this situation, which is the
requirement of invertibility of the symbol of the operator. It was proved that
the  two ellipticity conditions (they correspond to the two definitions of the
symbol) are equivalent under quite general assumptions.  Ellipticity implies
Fredholm property of the operator in suitable Sobolev spaces of type $H^s$.

Let us note here one essential difference between the theory of elliptic
$G$-$\psi$DO and a similar theory of $\psi$DO. Namely, examples show
(see~\cite{AnBr1,AnLe2,Ross1}) that the ellipticity (and  the Fredholm
property of   operator \eqref{eq-ooper1}  in the Sobolev spaces $H^s$)
essentially depends on the smoothness exponent    $s$. Thus, there arise
natural questions on the description of the possible values of $s$, for which
a given $G$-operator is elliptic and the question about the dependence of the index on   
$s$.  The answers to these questions are well known in the situation of an
isometric action of the group, that is, if the diffeomorphisms preserve a
Riemannian metric on the manifold. In this case the symbol and the index do not
depend on  $s$.  First steps in the study of these questions for nonisometric
actions were done in the papers   \cite{SaSt25,Sav12}, where it was shown for
the simplest nonisometric diffeomorphism of dilation of   spheres that the set
of  $s$, for which a $G$-operator is elliptic, is always an interval and the index
(inside this interval) does not depend on   $s$.

\vspace{1mm}

\noindent\textbf{2. Index of elliptic operators  for a discrete group of
diffeomorphisms.} Let us now turn   attention to the problem of computing the
index of elliptic $G$-operators. The first formula for the index of
$G$-operators was obtained in the paper~\cite{Ant4} for a  {\em finite group}
$G$ of diffeomorphisms.\footnote{These results were rediscovered in \cite{Park1}.} 
In this case the index of a  $G$-operator was expressed
in terms of Lefschetz numbers of an auxiliary elliptic $\psi$DO on   $M$.
Since the Lefschetz numbers are expressed by a formula    \cite{AtSi3} similar to the
Atiyah--Singer index formula and, hence,   the index problem for a
finite group is thus solved.

The index problem for infinite groups turned out to be much more  difficult and
required application of new methods related with noncommutative geometry of
Connes~\cite{Con6,Con1}. The first advance was done in the celebrated work of
Connes~\cite{Con4}. There  an index formula was obtained for operators of the
form
\begin{equation}\label{eq-con1}
    D=\sum_{\alpha\beta} a_{\alpha\beta} x^\alpha (d/dx)^{\beta}
\end{equation}
acting on the real line, where the coefficients $a_{\alpha\beta}$ are Laurent polynomials in the operators
$$
 (Uf)(x)=e^{ix}f(x) , \quad (Vf)(x)=f(x-\theta),
$$
and  $\theta$ is some fixed number. An index theorem of Connes for such
differential-difference operators  is naturally
formulated in terms of noncommutative geometry. Operators  
\eqref{eq-con1}, which are also called operators on the noncommutative
torus\footnote{This name is motivated by the fact that the algebra
generated by   $U$ and $V$ is a noncommutative deformation of the algebra of
functions on the torus   $\mathbb{T}^2$.}, were used in a mathematical
explanation of the quantum Hall effect  \cite{Con1}. It became clear after the
cited papers of Connes that noncommutative geometry is not only useful, but
also natural in the index problem for $G$-operators, and since then
noncommutative geometry is used in all the papers on the index of
$G$-operators, we are aware of. For instance, methods of noncommutative
geometry were applied to solve the index problem for deformations of
algebras of functions on toric manifolds in   \cite{LaSu1,CoDu1,CoLa2} and
other papers.

Further progress in the solution of the index problem for    $G$-operators was made
 in the monograph~\cite{NaSaSt17}. 
Namely, an index formula for operators   \eqref{eq-ooper1}  was obtained in the
situation, when the action is isometric. Let us note here that this index formula for isometric actions contains all
the above mentioned formulas as special cases.

In the situation of a general (that is,  {\em nonisometric}) action there were no index
formulas  until recently. There were only partial results. Namely,
  the index problem for $\mathbb{Z}$-operators (that is, operators for the group of integers)
was reduced to a similar problem for an elliptic $\psi$DO   (see
\cite{Ant2,Sav9,Sav10}). The first index formula in the nonisometric case was
obtained in  the paper\cite{SaSt25} for operators associated with dilation
diffeomorphism of spheres. The index formula for elliptic operators associated
with the group $\mathbb{Z}$ was obtained in \cite{SaSt30}.  Finally, an  index formula for     an arbitrary torsion free group acting  on the circle  was
stated in   \cite{Per4}.

We also  mention  several interesting examples of elliptic   $G$-operators.
Suppose that $G$ preserves some geometric structure on the manifold   (for
instance, Riemannian metric, complex structure, spin structure, ...). Then we
can consider an elliptic operator associated with that structure and twist this
operator using a   $G$-projection (that is an operator of the form
\eqref{eq-ooper1}, which is a projection: $P^2=P$) or an invertible $G$-operator. This construction produces an
elliptic   $G$-operator. For instance, if $G$ acts isometrically, then one can
take classical geometric operators    (Euler, signature,  Dolbeault, Dirac
operators). The indices of the corresponding twisted $G$-operators were
computed in   \cite{SaSt23}.  If $G$ acts by conformal diffeomorphisms of a
Riemannian surface, then one can take the $\overline{\partial}$ operator.
Indices of the corresponding twisted operators were computed in
\cite{Per2,Per3}. In the papers  \cite{CoMo2,Mos2} there are index formulas for
the twisted Dirac operator for group actions preserving the conformal structure
on the manifold.

\vspace{2mm}

\noindent\textbf{3. Operators associated with compact Lie groups.}
Let now $G$ be a {\em compact Lie group } acting on $M$. Consider the class of operators of the form
\begin{equation}\label{eq-ooper2}
    D=\int_G D_gT_g dg: C^\infty(M)\lra C^\infty(M)
\end{equation}
(cf.  \eqref{eq-ooper1}), where $dg$ is the Haar measure. Such operators
related values of functions on submanifolds of $M$ of positive dimension.  Such
operators were considered in  \cite{SaSt22,Ster20,Sav11}. In these papers a $G$-operator
of the form \eqref{eq-ooper2} was represented as a pseudodifferential operator acting in sections of
infinite-dimensional bundles~\cite{Luk1}, whose fiber is the space of functions
on     $G$. This method goes back to the papers of Babbage~\cite{Bab1} and for
a finite group gives a finite system of equations~\cite{Ant2}. Moreover, the
obtained operator, which we denote by $\mathcal{D}$, is  $G$-invariant, and its
restriction $\mathcal{D}^G$ to the subspace of  $G$-invariant functions is
isomorphic to the original operator     $D$. Now if
$\widehat{\mathcal{D}}=1+\mathcal{D}$ is transversally elliptic\footnote{This
notion was introduced by Atiyah and Singer \cite{Ati8,Sin3} and actively
studied since then (see especially~\cite{Kord4,Kord5,Kord2} and the references
cited there).} with respect to the action of  $G$, then this implies the
Fredholm property, that is,  the index of the operator    $\widehat{D}=1+D$ is
finite. The index formula and the corresponding topological invariants of the
symbol of elliptic $G$-operators were computed in the papers cited above.

\vspace{2mm}

\noindent\textbf{4. Other classes of $G$-operators.}
Operators associated with diffeomorphisms are not exhausted by operators of the form  \eqref{eq-ooper1}. In this section we consider   other classes of operators appearing in the literature.  Boundary value problems  similar to Carleman's problem, with the boundary condition relating the values of the unknown function at different points on the boundary were considered  (see monograph of Antonevich, Belousov and Lebedev  \cite{AnLe2} and the references cited there). Finiteness theorems were proved and for the case of finite group actions index theorems were obtained   (see also \cite{SaSt10}). On the other hand, nonlocal boundary value problems,  in which the boundary condition relates the values of a function on the boundary of the domain and on submanifolds, which lie inside the domain  there were considered in    \cite{BiSo3,Sku1,Sku2,Sku3}.   We also mention that  $G$-operators on manifolds with singularities were   considered in\cite{AnLe2}. The symbol was defined and a finiteness theorem was proved.

An important extension of the notion of (Fredholm)  index was obtained in
\cite{MiFo1}. Namely, given a   $C^*$-algebra $A$ (the algebra of scalars) one
considers operators $F$ acting on the spaces, which are $A$-modules. The index
of a Fredholm operator in this setting, also called Mishchenko--Fomenko index
\begin{equation}\label{eq-indmf1}
    \ind_A F\in K_0(A)\vspace{-2mm}
\end{equation}
is an element of the $K$-group of $A$.  Also, in the cited paper a  definition
of pseudodifferential operators over $C^*$-algebras was given and an index
theorem was proved. Note, however, that in applications it is sometimes useful
to have  not only the index \eqref{eq-indmf1} but some {\em numerical }
invariants. Such invariants can be constructed using the approach of
noncommutative geometry by pairing the index  \eqref{eq-indmf1} with cyclic
cocycles over $A$. In the papers
\cite{NaSaSt17,SaSt20,SaSt26} $G$-operators over $C^*$-algebras were defined for isometric actions  and  the finiteness theorem and the index
formula were obtained.

\vspace{2mm}

\noindent\textbf{5. Methods used in the theory of    $G$-operators.}
 Let us know write a few words about the methods used in obtaining these index formulas.
 The first approach, which appears naturally, is to try to adapt the known methods
 of obtaining index formulas for $\psi$DOs in our more general setting of  $G$-operators.
 This approach was successfully applied, for instance, in the book   \cite{NaSaSt17}. 
Note, however, that using  this approach we obtain the proof of the  index
formula, which  is quite nontrivial and relies on serious mathematical results,
notions and constructions from  noncommutative geometry and algebraic
topology.

The second approach   uses the  idea of
uniformization \cite{SaSt22,SaSt30} (see also \cite{SaSchSt1,SaSchSt2,SaSchSt3}) to  reduce the index problem for a   $G$-operator to a similar problem for a   {\em pseudodifferential\/} operator on a manifold of a higher dimension.
The index of the latter operator can be found using the celebrated Atiyah--Singer formula.
The attractiveness of this approach is based on the fact that this approach is quite
elementary and does not require application of complicated mathematical apparatus, which
was mentioned above.  This method of   {\em pseudodifferential uniformization} enabled to
give simple and elegant index formulas.

\vspace{3mm}

Let us now describe the contents of the remaining sections of the paper. In Section 1  we recall the definitions of symbol and the finiteness theorem for $G$-operators associated with actions of discrete groups. Section 2 is devoted to index formulas for actions of discrete groups. We start with the index formula for isometric actions and then give an index formula for nonisometric actions. Finally, Section 3 is devoted to $G$-operators associated with compact Lie group actions. We show how pseudodifferential uniformization can be used to obtain a finiteness theorem for such operators.


\section{Elliptic operators associated with actions of discrete groups}

\subsection{Main definitions}

Let $M$ be a closed smooth manifold and $G$ a discrete group acting on  $M$
by diffeomorphisms. We  consider the class of operators of the form
\begin{equation}\label{op-shift1}
D=\sum_{g\in G} D_gT_g:C^\infty(M)\lra C^\infty(M),
\end{equation}
where $\{D_g\}_{g\in G}$ is a collection of pseudodifferential operators of order $\le m$
acting on   $M$. We suppose that only finitely many   $D_g$'s are nonzero.
Finally, $\{T_g\}$ stands for the representation of $G$ by the shift operators
$$
(T_gu)(x)=u(g^{-1}(x)).
$$
Here and below an element $g\in G$  takes a point $x\in M$  to the point denoted by
 $g(x)\in M$.

Main problems:
\begin{enumerate}
\item Give  {\em ellipticity conditions}, under which the operator
\begin{equation}\label{op-shift1s}
D:H^s(M)\lra H^{s-m}(M),\quad m={\rm ord}\; D.
\end{equation}
is Fredholm in the Sobolev spaces.
\item  Compute the index of operator \eqref{op-shift1s}.
\end{enumerate}
The first of these problems is treated in this section, while the second problem is treated in the subsequent section.

Below operators of the form \eqref{op-shift1} are called {\em $G$-pseudodifferential operators} or {\em $G$-operators} for short.

\subsection{Symbols of operators}

\paragraph{Definition of symbol.}

The action of $G$ on $M$ induces a representation of this group by
automorphisms of the algebra   $C(S^*M)$ of continuous symbols on the cosphere
bundle  $S^*M=T^*_0M/\mathbb{R}_+$. Namely, an element $g\in G$ acts as a shift
operator along the trajectory of the mapping   $\partial g:S^*M\to S^*M$, which
is the extension of $g$  to the cotangent bundle and is defined as  $\partial
g=({}^t(dg))^{-1}$, where $dg:TM\to TM$ is the differential.  Consider the
$C^*$-crossed product $ C(S^*M)\rtimes G$ (e.g., see \cite{Bla1,Ped1,Wil1,Zell1}) of
the algebra $C(S^*M)$ by the action of  $G$. Recall that $ C(S^*M)\rtimes G$ is
the algebra, obtained as a completion of the algebra of compactly-supported
functions on    $G$ with values in  $C(S^*M)$ and the product of two elements
is defined as:
$$
a b(g)=\sum_{kl=g} a(k){k^{-1}}^*(b(l)),\quad k,l\in G.
$$
The completion is taken with respect to a certain norm.\footnote{Below we
consider the so-called maximal crossed product.}  Here for  $k\in G$ by
${k^{-1}}^*:C(S^*M)\to C(S^*M)$ we denote the above mentioned automorphism of  $C(S^*M)$.

To define  the symbol for $G$-operators, it is useful to replace the
shift operator   $T_g:H^s(M)\lra H^s(M)$ by the following unitary operator.
We fix a smooth positive density $\mu$ and a Riemannian metric   on  $M$ and
treat $H^s(M)$ as a Hilbert space with the norm
$$
\|u\|^2_{H^s}=\int_M |(1+\Delta)^{s/2}u|^2 \mu,
$$
where $\Delta$ is the nonnegative Laplacian. A direct computation shows that the operator
$$
{T}_{g,s}=(1+\Delta)^{-s/2}\mu^{-1/2}T_g\mu^{+1/2}(1+\Delta)^{s/2}:H^s(M)\lra
H^s(M)
$$
is unitary. Here
$$
 \mu^{1/2}:L^2(M)\to L^2(M,\Lambda^{1/2})
$$
is the isomorphism of $L^2$ spaces of scalar functions and half-densities on  $M$  defined by multiplication by the square root of    $\mu$.
Note that the operator ${T}_{g,s}$ can be decomposed as
$ {T}_{g,s}=A_{g,s}T_g$, where $A_{g,s}$ is an invertible elliptic $\psi$DO of order zero.

This implies that the class of operators   \eqref{op-shift1} will not
change if in  \eqref{op-shift1} we replace  $T_g$ by  $ {T}_{g,s}$. Now we can give the definition of the symbol.
\begin{definition}\label{def2}
The \emph{symbol} of operator
\begin{equation}\label{eq-shiftss}
    D=\sum_{g\in G}  {D}_g {T}_{g,s}:H^s(M)\lra H^{s-m}(M),
\end{equation}
where $\{ {D}_g\}$ are pseudodifferential operators of order $\le m$ on  $M$, is an element
\begin{equation}\label{eq-symbol1s}
\sigma_s(D)\in C(S^*M)\rtimes G,
\end{equation}
defined by the equality $\sigma_s(D)(g)=\sigma( {D}_g)$ for all $g\in G$.
\end{definition}

The symbol  \eqref{eq-symbol1s} is not completely convenient for applications,
since it depends on the choice of  $\Delta$ and  
$\mu.$ Here we give another definition of the symbol which is free from this drawback.

\paragraph{Trajectory symbol.}
So, let us try to define the symbol of the operator   \eqref{op-shift1s} using
the method of frozen coefficients. Note that the operator is essentially nonlocal.
More precisely, the corresponding equation   $Du=f$ relates values of the unknown
function   $u$ on the   \emph{orbit} $Gx_0\subset M$, rather than  at a single point $x_0\in M$.
For this reason, unlike the classical situation, we need to freeze the coefficients
 of the operator on the entire orbit of    $x_0$. Freezing the
 coefficients of the operator   \eqref{op-shift1s} on the orbit of $x_0$ and applying
 Fourier transform   $x\mapsto \xi$, we can define  the symbol as a function on the
 cotangent bundle    $T^*_0M=T^*M\setminus 0$
with zero section deleted. This function ranges in operators acting on the space of
functions on the orbit. A direct computation gives the following expression for the symbol   (see \cite{Ant2,Sav12}):
\begin{equation}\label{traj-symbol1}
\sigma(D)(x_0,\xi)=\sum_{h\in G}
\sigma(D_h) (g^{-1}(x_0),\partial g^{-1}(\xi))\mathcal{T}_h:l^2(G,\mu_{x_0,\xi,s})\longrightarrow
l^2(G,\mu_{x_0,\xi,s-m}).
\end{equation}
Here we identify the orbit     $Gx_0$ with the group $G$ using the mapping $g(x_0)\mapsto g^{-1}$
and use the following notation:
\begin{itemize}
\item $(\mathcal{T}_hw)(g)=w(g h)$ is the right shift operator on the group;

\item the expression $
 \sigma(D_h) (g^{-1}(x_0),\partial g^{-1}(\xi))$ acts as an operator of multiplication of functions on the group;

\item the space $l^2(G,\mu_{x,\xi,s})$ consists of functions  $\{w(g)\}$,
$g\in G$, which are square summable with respect to the density $\mu_{x_0,\xi,s}$,
which in local coordinates is defined by the expression \cite{Sav12}
\begin{equation}\label{measure-coord1}
\mu_{x_0,\xi,s}(g)=\left|\det \frac{\partial g^{-1}}{\partial x}\right|\cdot
      \left|{  \Bigl.\Bigr.^t}{\left(\frac{\partial g^{-1}}{\partial x}\right)}^{-1}(\xi)\right|^{2s}
\end{equation}
More precisely, here we suppose that the manifold is covered by a finite number
of charts and the diffeomorphism   $g^{-1}$ is written (in some pair of charts)
as $x\mapsto g^{-1}(x)$. The density is unique (up to equivalence of
densities).
\end{itemize}
\begin{definition}\label{def1} The operator \eqref{traj-symbol1} is the {\em
trajectory symbol} of operator \eqref{op-shift1s} at $(x_0,\xi)\in T^*_0M$.
\end{definition}
Note that in general, the dependence of the trajectory symbol on  $x_0,\xi$
is quite complicated.  For instance, the symbol may be discontinuous.   This is related
with the fact that the structure of the orbits of the action can be quite complicated.

Let us describe the relation between the symbols defined in Definitions   \ref{def1} and \ref{def2}.
Given  $(x,\xi)\in S^*M$, we define the representation
$$
\begin{array}{ccc}
   \pi_{x,\xi}:C(S^*M)\rtimes G &\lra &\mathcal{B}l^2(G)\\
   f &\longmapsto & \sum_h f(g^{-1}(x),\partial g^{-1}(\xi),h)\mathcal{T}_h
\end{array}
$$
of the crossed product in the algebra of bounded operators acting on the
standard space   $l^2(G)$ on the group (cf. \eqref{traj-symbol1}). One can show
that the    diagram 
\begin{equation}\label{eq-diag1}
\xymatrix{ l^2(G,\mu_{x,\xi,s})\ar[rr]^{\sigma(D)(x,\xi)} \ar[d]^\simeq& &
l^2(G,\mu_{x,\xi,s-m}) \ar[d]^\simeq\\
l^2(G) \ar[rr]_{\pi_{x,\xi}(\sigma_s(D))} & & l^2(G),}
\end{equation}
commutes, where the vertical mappings are isomorphisms defined by multiplication
by the square root of the densities. In other words, this commutative diagram
shows that the restriction of the symbol   $\sigma_s(D)$ to a trajectory gives the trajectory symbol $\sigma(D)$.

\subsection{Ellipticity and Finiteness theorem}

The  two definitions of the symbol give two notions  of ellipticity.

\begin{definition}\label{d1}
Operator \eqref{op-shift1s} is \emph{elliptic}, if its trajectory symbol \eqref{traj-symbol1} is invertible on $T^*_0M$.
\end{definition}
\begin{definition}\label{d2}
Operator \eqref{op-shift1s} is called \emph{elliptic}, if its symbol
  \eqref{eq-symbol1s} is invertible as an element of the algebra $C(S^*M)\rtimes G$.
\end{definition}
It turns out that these definitions of ellipticity are equivalent, at least for a quite large class of groups.  More precisely, the commutative diagram   \eqref{eq-diag1} shows that ellipticity in the sense of Definition~\ref{d2} implies ellipticity in the sense of Definition~\ref{d1}; the inverse assertion is more complicated and was proved in   \cite{AnLe1} for actions of amenable groups (recall that a discrete group $G$ is amenable, if there is a $G$-invariant mean on $l^\infty(G)$; for more details see, e.g.,  \cite{Pat2}). We suppose that below all groups are amenable and we identify these two notions of ellipticity.

The following finiteness theorem is proved by standard techniques  (see \cite{AnLe1,AnLe2}).
\begin{theorem}\label{th-fredholm}
If   operator \eqref{op-shift1s} is elliptic then it is Fredholm.
\end{theorem}

\begin{remark}
It is shown in the cited monographs   \cite{AnLe1,AnLe2} that under certain
quite general assumptions   (namely,  the action of   $G$ on  $M$ is
assumed to be topologically free, that is, for any finite set $\{g_1,...,g_n\}\subset G\setminus\{e\}$ the union $  M^{g_1}\cup...\cup M^{g_n}$ of the fixed point sets has an empty interior), the ellipticity condition is necessary for the Fredholm property. If the
action is not topologically free, then one could give a finer ellipticity
condition. We do not consider these conditions here and refer the reader to the
monograph   \cite{AnLe1}.
\end{remark}

\subsection{Examples}
Let us   illustrate the   notion of ellipticity for $G$-operators on several explicit examples.

{\bf 1. Operators for the irrational rotations of the circle.}
Consider the group $\mathbb{Z}$ of rotations of  the circle $\mathbb{S}^1$
by the angles multiples of a fixed angle    $\theta$ not commensurable to $\pi$:
$$
g(x)=x+g\theta, \quad x\in\mathbb{S}^1,\quad g\in\mathbb{Z},\quad \theta\notin \pi \mathbb{Q}.
$$
A direct computation shows that in this case the densities $\mu_{x,\xi,s}$
(see \eqref{measure-coord1}) are equivalent to the standard density $\mu(g)=1$ on the lattice $\mathbb{Z}$.
Hence, in this case the symbol of the operator
$D=\sum_{g\in\mathbb{Z}}D_gT_g$ is equal to
$$
\sigma(D)(x,\xi)=\sum_h
\sigma(D_h)(x-g\theta,\xi)\mathcal{T}^h:l^2(\mathbb{Z})\to l^2(\mathbb{Z}),\quad \text{where } \mathcal{T}u(g)=u(g-1).
$$
Let us make two remarks. First, in this example, as in the classical theory of $\psi$DOs,
the symbol does not depend on   $s$ and therefore an operator is elliptic or not elliptic for all
 $s$ simultaneously. The same property holds in the general case if the action is isometric.
 Second, in this case to check the ellipticity condition, it suffices to check that the symbol
 is invertible only for one pair of points   $(x_0,\pm 1)$. Indeed, since
 $S^*\mathbb{S}^1=\mathbb{S}^1\cup \mathbb{S}^1$, the crossed product   $C(S^*\mathbb{S}^1)\rtimes\mathbb{Z}$
is a direct sum of two simple algebras\footnote{That is, algebras without nontrivial ideals.}
of irrational rotations $C(\mathbb{S}^1)\rtimes\mathbb{Z}$. Hence, the mapping
$$
 \pi_{x_0,1}\oplus \pi_{x_0,-1}:C(S^*\mathbb{S}^1)\rtimes\mathbb{Z}\lra \mathcal{B}l^2(\mathbb{Z})\oplus  \mathcal{B}l^2(\mathbb{Z})
$$
is  a monomorphism. Therefore,  the symbol $\sigma(D)$ is invertible if and
only the trajectory symbols at the points   $(x_0,\pm  1)$ are invertible.

{\bf 2. Operators for dilations of the sphere \cite{SaSt25}.}
On the sphere $\mathbb{S}^m$ we fix the North and the South poles.
The complements of the poles are identified with   $\mathbb{R}^m$ with
the coordinates $x$ and $x'$, correspondingly. Let us choose the following
transition function  $x'(x)=x|x|^{-2}$. Consider the action of $\mathbb{Z}$
on $\mathbb{S}^m$, which in the $x$-coordinates is generated by the dilations   
$$
g(x)=\alpha^gx,\quad g\in\mathbb{Z},x\in\mathbb{R}^m,
$$
where $\alpha$
($0<\alpha<1$)
is fixed. This expression defines a smooth action on the sphere. Let us compute the densities
$\mu_{x,\xi,s}$.
\begin{proposition}\label{prop1}
Depending on whether   $x$  is a pole of the sphere or not,   the density $\mu_{x,\xi,s}$ in \eqref{measure-coord1} is equal to:
$$
 \mu_{x,\xi, s}(g)=\left\{
\begin{array}{ll}
\alpha^{|g|(m-2s)}, & \text{if  }x\ne0,x\ne \infty,\\
\alpha^{ g (m-2s)}, & \text{if  }x=0,\\
\alpha^{ -g (m-2s)}, & \text{if  }x=\infty.\\
\end{array}
\right.
$$
\end{proposition}
\begin{proof}
Indeed, given  $g\le 0$ the points $g^{-1}(x)$ remain in a bounded domain of the chart
$\mathbb{S}^m\setminus \infty$. Thus, we can apply the formula
\eqref{measure-coord1}, in which we use the  $x$-coordinate in the domain and the range of the diffeomorphism
 $g$. We get $\partial g^{-1} /\partial
x=\alpha^{-g} I$. Hence
$$
\mu_{x,\xi,s}(g)=\left|{  \Bigl.\Bigr.^t}{\left(\frac{\partial g^{-1}}{\partial x}\right)}^{-1}(\xi)\right|^{2s}=\alpha^{-gm} \cdot |\xi/\alpha^{-g}|^{2s}=
      \alpha^{-g(m-2s)}|\xi|^{2s}.
$$
This gives the desired expression for the measure if  $g \le 0$. Now if $g\to +\infty$ then the points
$g^{-1} (x)=\alpha^{-g}x$ tend to infinity and we can apply  formula
\eqref{measure-coord1}, where we use the pair of coordinate charts   $x$ and  $x'$.
A computation similar to the previous one gives the desired expression for the measure at the poles of the sphere:
$x=0$ and  $x =\infty$.
\end{proof}
Consider the operator
\begin{equation}\label{eq-x1}
 D=\sum_k D_kT^k:H^s(\mathbb{S}^m)\longrightarrow H^s(\mathbb{S}^m),\quad Tu(x)=u(\alpha^{-1}x).
\end{equation}
According to the obtained expressions for the densities, this operator has the
symbol $\sigma(D)(x,\xi)$ at each point $(x,\xi)\in T^*_0\mathbb{S}^m$. For
example,  consider the point   $x=0$. It follows from Proposition~\ref{prop1}
that we obtain an expression for the symbol at this point
\begin{equation*}\label{traj-symbol1z}
\sigma(D)(0,\xi)=\sum_k  \sigma(D_k)
(0,\xi)\mathcal{T}^k:l^2(\mathbb{Z},\mu_{s})\longrightarrow
l^2(\mathbb{Z},\mu_{s}),\quad \mu_s(n)=\alpha^{-n(m-2s)}.
\end{equation*}
The Fourier transform
$\{u(g)\}\mapsto \sum_g u(g)w^{-g}$ takes the latter operator to the operator of  multiplication
\begin{equation*}\label{eq-sym-s}
\sigma_S(D)(\xi,w)=\sum_k \sigma(D_k)\left(0,\xi \right)w^k:
L^2(\mathbb{S}^1)\longrightarrow L^2(\mathbb{S}^1),\quad \xi\in
\mathbb{S}^{m-1}, \quad |w|=\alpha^{-m/2+s}
\end{equation*}
by a smooth function on the circle  $\mathbb{S}^1$ of radius $\alpha^{-m/2+s}$.
This shows that in this example the ellipticity condition explicitly depends on
the smoothness exponent $s$. It was proved inn  \cite{SaSt25}  that  the set of
values of $s$ for which  the operator \eqref{eq-x1} is elliptic is an open
interval (possibly (semi)infinite or empty).

\section{Index formulas for actions of discrete groups}

 In the previous  section we defined the symbol for of a $G$-operator  as an element  of the corresponding crossed product.  If an operator   $D$ elliptic  (its symbol is invertible) then $D$ has Fredholm property and its index   $\ind D$ is defined. To solve the index problem means to  express the index in terms of the symbol of the operator and the topological characteristics  of the  $G$-manifold.

\subsection{Isometric actions}

The index problem for $G$-operators was solved in    2008 for isometric actions
in \cite{NaSaSt17}.  Here we discuss the index formula from the cited monograph.
This formula is proved under the following assumption.

\begin{assumption}
\begin{enumerate}
\item $\Ga$ is a discrete group of  {\em polynomial growth} (see \cite{Gro1}),
i.e.,  the number of elements of the group, whose length is   $\le N$ in the
word metric on the group grows at most as a polynomial in $N$ as  $N\to \infty$. 
\item
$M$ is a Riemannian manifold  and the action of $G$ on $M$ is isometric.
\end{enumerate}
\end{assumption}

\paragraph{Smooth crossed product.}
Let  $D$ be an elliptic operator. Then its symbol is invertible and defines an element
$$
[\sigma(D)]\in K_1(C (S^*M)\rtimes G)
$$
of the odd $K$-group of the crossed product $C (S^*M)\rtimes G$ (e.g., see
\cite{Bla1}). Note a significant difference  between the   elliptic
theory of $G$-operators and the classical Atiyah--Singer theory: the algebra of symbols is not
commutative and therefore we use $K$-theory of algebras instead of topological
$K$-theory. Further, to give an index formula, we will use tools from
noncommutative differential geometry. Note  that noncommutative differential
geometry does not apply in general to  $C^*$-algebras. The point here is that
in a  $C^*$-algebra  there is a notion of continuity, but there is no
differentiability.  Fortunately, in the situation at hand, one can prove that
we only deal with differentiable elements. Let us formulate this
statement precisely.
\begin{proposition}[see \cite{Schwe1}]
If the symbol $\sigma(D)$ is invertible, then the inverse $\sigma(D)^{-1}$  lies in the subalgebra
\begin{equation}\label{eq-smooth1}
C^\infty(S^*M)\rtimes G\subset C(S^*M)\rtimes G,
\end{equation}
of $C^\infty(S^*M)$-valued functions on $G$, which (together  with all their
derivatives) tend to zero as  $|g|\to\infty$ faster than an arbitrary power of
$|g|$.
\end{proposition}
The subalgebra  \eqref{eq-smooth1} is called the {\em smooth crossed product}.

So, we have
\begin{equation}\label{eq-kth1}
[\sigma(D)]\in K_1(C^\infty (S^*M)\rtimes G).
\end{equation}

To write an index formula for   $D$, we first define a topological  invariant
of the symbol. This invariant is called the Chern character of the element
\eqref{eq-kth1}. Then we define a topological invariant of the manifold.

\paragraph{Equivariant Chern character.}
Following \cite{NaSaSt17}, let us define  the Chern character as the homomorphism of groups
\begin{equation}\label{eq-chch1}
\ch: K_1(C^\infty(X)\rtimes G)\lra \bigoplus_{\langle g\rangle\subset G} H^{odd}(X^g),
\end{equation}
where we put for brevity $X=S^*M$, the sum runs  over conjugacy classes of   $G$,
and $X^g$ denotes the fixed-point set of $g$. Since  $g$  is an isometry by
assumption, the fixed-point set is a smooth submanifold
  (e.g., see~\cite{CoFl1}).

We define  the Chern character using the abstract approach of noncommutative
geometry. To this end, it suffices to define  a pair $(\Omega,\tau)$, where:
\begin{enumerate}
\item   $\Omega=\Omega_0\oplus\Omega_1\oplus\Omega_2\oplus\ldots$  is a
differential  graded algebra, which contains the crossed product
$C^\infty(X)\rtimes G$ as a subalgebra of  $\Omega_0$; 
\item   $\tau:\Omega\lra
\bigoplus_{\langle g\rangle\subset G} \Lambda(X^g)$ is a homomorphism
of differential complexes such that
\begin{equation}\label{eq-sled}
 \tau (\om_2\om_1)=(-1)^{\deg\om_1\deg\om_2}\tau  (\om_1\om_2),\qquad  \text{for all }\om_1,\om_2\in\Om.
\end{equation}
\end{enumerate}
The algebra $\Omega$ is called the   {\em algebra of noncommutative differential
forms},  and the functional  $\tau$ is called the {\em differential graded trace}.

If such a pair is given, then the Chern character  associated with the pair
$(\Omega,\tau)$ is defined as
\begin{equation}\label{eq-chern-odd}
\ch (a)= \tr\tau \left[
 \sum_{n\ge 0}\frac{n!}{(2\pi i)^{n+1}(2n+1)!}
(a^{-1}da)^{2n+1} \right],\quad [a]\in K_1(C^\infty(X)\rtimes G),
\end{equation}
where $\tr$ is the trace of a matrix. A standard computation shows that   the
form in \eqref{eq-chern-odd}  is closed and its class in de Rham cohomology is
determined by   $[a]$ and defines the homomorphism \eqref{eq-chch1}. It remains
to define the pair $(\Omega,\tau)$:

1. We set   $\Omega=\Lambda(X)\rtimes \Ga$, where the differential on the
smooth crossed product of the algebra $\Lambda(X)$ of differential forms on $X$ and the group  $G$
is equal to
$$
(d\om)(g)=d(\om(g)),\qquad \om\in \Lambda(X)\rtimes \Ga.
$$

2. To define a differential graded trace  $\tau=\{\tau_g\}$, we fix some
$g\in\Ga$ and introduce necessary notation. Let   $\overline{G}$ be the closure
of $G$  in the compact Lie group of isometries of   $X$. This closure is a
compact Lie group. Let   $C_g\subset \overline{G}$ be the
centralizer\footnote{Recall that the centralizer of   $g$ is the subgroup of
elements commuting with $g$.} of $g$. The centralizer is a closed Lie subgroup
in   $\overline{G}$.  Denote the elements of the centralizer by   $h$, and the
induced smooth Haar measure on the centralizer by  $dh$.

Let $\langle g\rangle\subset\Ga$ be the conjugacy class of   $g$, i.e., the
set of elements equal to $zgz^{-1}$ for some $z\in \Ga$. Further, for each $g'\in
\langle g\rangle$ with fix some element $z=z(g,g')$, which conjugates $g$ and
$g'=zgz^{-1}$. Any such element defines a diffeomorphism   $z:X^g\to X^{g'}$.

Let us define the trace as
\begin{equation}\label{eq-sled-nash}
    \tau_g(\om)=
      \sum_{g'\in \langle g\rangle}\;\;\;
        \int_{C_g}
           \Bigl.h^*\bigl(
              {z}^*\om(g')
             \bigr)\Bigr|_{X^g}
           dh,\quad \text{where }\omega\in \Lambda(X)\rtimes G.
\end{equation}
One can show that this expression does not depend on the choice of elements $z$ and is indeed a differential graded trace.

\begin{remark}
For a finite group the Chern character~\eqref{eq-chern-odd}
coincides with the one constructed in \cite{Slo1},~\cite{BaCo2}.
\end{remark}

\paragraph{Equivariant Todd class.}

Given $g\in\Ga$, the normal bundle of the fixed-point  submanifold $M^g\subset
M$ is denoted by $N^g$. The differential $dg$ defines an orthogonal
endomorphism of  $N^g$ and the corresponding bundle of exterior forms
$$
 \Omega(N^g_\mathbb{C})=\Omega^{ev}(N^g_\mathbb{C})\oplus
 \Omega^{odd}(N^g_\mathbb{C}).
$$
Here $E_\mathbb{C}$ stands for the complexification of a real vector bundle $E$.
Consider the expression  (see \cite{AtSi3})
\begin{equation}\label{eq-forms}
\ch \Omega^{ev}(N^g_\mathbb{C})(g)-\ch
\Omega^{odd}(N^g_\mathbb{C})(g)\in H^{ev}(M^g).
\end{equation}
The zero-degree component of this expression is nonzero  \cite{AtSe2}.  Hence the class \eqref{eq-forms}
is invertible and the following expression is well-defined
\begin{equation}\label{eq-todd-class}
    \Td_g(T^*_\mathbb{C}M )=\frac{\Td(T^*_\mathbb{C}M^g )}{\ch \Omega^{ev}(N^g_\mathbb{C})(g)-\ch
\Omega^{odd}(N^g_\mathbb{C})(g)}\in H^*(M^g),
\end{equation}
where $\Td$ on  the right-hand side in the equality is the Todd class of a
complex vector bundle, and the expression is well-defined, since the forms have even
degrees.

\paragraph{Index theorem}

\begin{theorem}[see \cite{NaSaSt17}]\label{atsi}
Let $D$ be an elliptic $G$-operator on a closed manifold $M$.
Then
\begin{equation}\label{Cohom1}
\ind D=\sum_{\langle g \rangle\subset\Ga}\left\langle
\ch_g[\sigma(D)]\Td_g(T^*_\mathbb{C}M ), [S^*M^{g}]\right\rangle,
\end{equation}
where $\langle g\rangle$ runs over the set of conjugacy classes of   $\Ga$;
$[S^*M^{g}]\in H_{odd}(S^*M^{g})$ is the fundamental class of  $S^*M^{g}$; the
Todd class is lifted from $M^{g}$ to  $S^*M^{g}$ using the natural projection; the brackets $\langle,\rangle$ denote the pairing of cohomology and
homology. The series in \eqref{Cohom1} is absolutely convergent.
\end{theorem}

In some situations the sum in \eqref{Cohom1}  can be reduced to   one
summand equal to the contribution of the unit element of the group.
\begin{corollary}[see \cite{NaSaSt17,SaSt23}]\label{th-e}
Suppose that either the action of  $\Ga$ on $M$ is free or
 $\Ga$ is torsion free. Then one has
\begin{equation}\label{eq-index-2}
    \ind D=\langle \ch_e[\sigma(D)]\Td(T^*_\mathbb{C}M ),[S^*M]\rangle.
\end{equation}
\end{corollary}

Let us note that the index formula   \eqref{Cohom1} contains many other index
 formulas as special cases   (see~\cite{NaSaSt17,SaSt23} for details).
 Here we give two situations, in which the index formula can be applied.

\paragraph{Example 1.  Index of twisted Toeplitz operators.}

Let $M$ be an odd-dimensional oriented manifold. We suppose that
  $M$ is endowed with a  $G$-invariant spin-structure (i.e., the action of  $G$
on $M$ lifts to an action on the spin  bundle $S(M)$). Let $\mathcal{D}$ be the
Dirac operator \cite{AtSi3}
\begin{equation*}\label{eq-dirac}
    \mathcal{D}:S (M)\lra S (M),
\end{equation*}
acting on spinors. This operator is elliptic and self-adjoint. Denote by
$\Pi_+:S (M)\lra S (M)$ the positive spectral projection of this operator.

We define  the Toeplitz operator
\begin{equation}\label{eq-dirac-twist}
    \Pi_+ U:  \Pi_+(S(M))\otimes \mathbb{C}^n\lra      \Pi_+(S(M))\otimes \mathbb{C}^n,
\end{equation}
where $U$ is an invertible  $n$ by $n$  matrix with elements in
$C^\infty(M)\rtimes\Ga$. Then the operator \eqref{eq-dirac-twist} is Fredholm
(its almost-inverse is equal  to  $\Pi_+ U^{-1}$). Let us suppose for
simplicity that  either   $\Ga$ is torsion free, or the action is free. In this
case, the formula
  \eqref{eq-index-2} gives the following expression for the index.
\begin{theorem}\label{th-dirac-index} The index of operator  \eqref{eq-dirac-twist}  is equal to
\begin{equation}\label{eq-dirac-index}
 \ind (\Pi_+ U )=\int_M A(TM)\ch_e(U),
\end{equation}
where  $A(TM)$ is the $A$-class of the tangent bundle, which in the Borel--Hirzebruch formalism is defined by the function
$$
\frac{x/2}{{\rm sh}\;x/2}.
$$
\end{theorem}

\paragraph{Examples 2. Operators on noncommutative torus.}

Let us fix $0<\theta\le 1$.  A. Connes in \cite{Con1} considered differential operators of the form
\begin{equation}\label{eq-operators-connes}
    D=\sum_{\alpha+\beta\le m} a_{\alpha\beta}x^\alpha\left(-i\frac
    d{dx}\right)^\beta:S(\mathbb{R})\lra S(\mathbb{R}),
\end{equation}
in the Schwartz space $S(\mathbb{R})$ on the real line. Here the coefficients $a_{\alpha\beta}$ are Laurent
polynomials in operators
$U,V$
\begin{equation}\label{eq-nctor-1}
    (Uf)(x)=f(x+1),\qquad (Vf)(x)=e^{-2\pi ix/\theta}f(x)
\end{equation}
of shift by one and product by exponential.

Let us show that the operators of the form  \eqref{eq-operators-connes}  reduce
to $G$-operators on a closed manifold. To this end, we consider the real line as
the total space of the standard covering
$$
\mathbb{R}\lra \mathbb{S}^1,
$$
whose base is the circle of length  $\theta$. Then the Schwartz space becomes
isomorphic to the space of smooth sections of a (nontrivial) bundle on the base
$\mathbb{S}^1$, whose fiber is the Schwartz space  $S(\mathbb{Z})$of rapidly
decaying sequences (that is, functions on the fiber). Then we apply Fourier
transform

$$
\mathcal{F}:S(\mathbb{Z})\lra C^\infty(\mathbb{S}^1)
$$
in each fiber and obtain a space, which is the space of smooth sections of  a
complex line bundle over the torus $\mathbb{T}^2$. These transformations define
the isomorphism
\begin{equation}\label{eq-isom1}
    S(\mathbb{R})\simeq C^\infty(\mathbb{T}^2,\gamma)
\end{equation}
of the Schwartz space on the real line and the space
\begin{equation*}\label{bottbundle}
C^\infty(\mathbb{T}^2,\gamma)=\{ g\in C^\infty(\mathbb{R}\times \mathbb{S}^1)
\; |\; g(\varphi+\theta,\psi)=g(\varphi,\psi)e^{-2\pi i\psi}\}.
\end{equation*}
of smooth sections of a complex line bundle  $\gamma$ on the torus. Here  on
$\mathbb{T}^2$ we consider the coordinates $0\le\varphi\le\theta$, $0\le\psi\le
1$. This isomorphism is defined by the formula
$$
f(x)\longmapsto \sum_{n\in\mathbb{Z}} f(\varphi+\theta n)e^{2\pi i n\psi}.
$$
Using the isomorphism \eqref{eq-isom1},  it is easy to obtain the correspondences between the operators:
\begin{equation*}\label{tab-correspondence}
    \begin{tabular}{|c|c|}
      \hline
     operators on the line & operators on the torus \\
      \hline
       $ -i\frac d{dx}$ & $-i\frac \partial{\partial\varphi}$\phantom{$\displaystyle\frac ZZ$}\\
      \hline $x$ &  $-i\frac\theta{2\pi}\frac \partial{\partial\psi}+\psi$ \phantom{$\displaystyle\frac ZZ$}\\
      \hline  $ e^{-2\pi ix/\theta} $ &  $  e^{-2\pi i\varphi/\theta} $\phantom{$\displaystyle\frac ZZ$} \\
      \hline $f(x)\to f(x+1)$& $g(\varphi,\psi)\mapsto g(\varphi+1,\psi)$\phantom{$\displaystyle\frac ZZ$}\\
      \hline
    \end{tabular}
\end{equation*}
This table implies that on the torus we obtain $G$-operators,  which
can be studied using the finiteness theorem and the index formula formulated above.
We refer the reader to  \cite{NaSaSt17} for details.

\subsection{General actions}

In this subsection we survey index formulas for elliptic operators associated
with general actions of discrete groups (see recent papers
\cite{SaSt30} and \cite{Per4}).  Let $D$ be an elliptic operator of the form
\eqref{op-shift1}. We will assume for simplicity that the inverse symbol
$\sigma(D)^{-1}$ lies in the algebraic crossed product
$$
 C^\infty(S^*M)\rtimes_{alg} G\subset C (S^*M)\rtimes G ,
$$
which consists of compactly supported functions on the group. Such a symbol defines an element
\begin{equation}\label{eq-kth2}
[\sigma(D)]\in K_1(C^\infty(S^*M)\rtimes_{alg} G)
\end{equation}
in $K$-theory. We would like to define the topological index as a  numerical
invariant associated with   $[\sigma(D)]$.  There is  a standard procedure in
noncommutative geometry of constructing such invariants. Namely,  one takes the pairing of   \eqref{eq-kth2}  with an element in {\em cyclic cohomology}  of the same algebra. Let us recall
this construction.

\paragraph{Cyclic cohomology. Pairing with $K$-theory.} Let $A$
be an algebra with unit. Recall  (see \cite{Con1}) that the  {\em cyclic
cohomology}  $HC^*(A)$ of $A$ is the cohomology of the bicomplex
\begin{equation}
\begin{array}{cccccc}
 &  &  & & A^{*} & \cdots\\
 & &  &  & \uparrow B &  \\
  &  & A^{* } & \stackrel{b}\to & A^{*2} & \cdots\\
 & & \uparrow B &  & \uparrow B &  \\
A^* & \stackrel{b}\to & A^{*2} & \stackrel{b}\to & A^{*3} & \cdots\\
\vdots &  & \vdots &  & \vdots &
\end{array}
\end{equation}
where $b,B$ are some differentials and for simplicity we denote the space of
multilinear functionals  on   $A^k$ by  $A^{*k}$.  In particular, an element
$\varphi\in HC^{n}(A)$ of cyclic cohomology is represented by a finite
collection of multi-linear functionals
$$
\{\varphi_j(a_0,...,a_j)\},\quad j=n,n-2,n-4,...,
$$
such that $B\varphi_j+b\varphi_{j-2}=0.$

{\small To make the paper self-contained, we recall the formulas for the
differentials in the bicomplex:
\begin{equation}\label{hoch-diff}
  \begin{split}
  (b\varphi)(a_0,a_1,\dotsc,a_{j+1})&=
  \sum_{n=0}^j
  \varphi(a_0,a_1,\dotsc,a_na_{n+1},\dotsc,a_{j+1})
  \\
  &\qquad{}+(-1)^{j+1}\varphi(a_{j+1}a_0,a_1,\dotsc,a_j).
  \end{split}
\end{equation}
and
$
B=Ns(Id-\lambda),
$
where $\lambda=(-1)^{n}(\text{cyclic left shift})$,
$$
s:A^{*(n+1)}\longrightarrow A^{*n},\quad
(s\varphi)(a_0,\ldots,a_{n-1})=\varphi(1,a_0,\ldots,a_{n-1}),
$$
and
$
N:A^{*n}\longrightarrow A^{*n},$ $N=Id+\lambda+\lambda^{2}+\ldots+\lambda^{n-1} $
is the symmetrization mapping.
}

The desired numerical invariants are defined using the pairing
\begin{equation}\label{eq-pairing1}
\langle,\rangle:K_1(A)\times HC^{odd}(A) \lra \mathbb{C}
\end{equation}
of $K$-theory and cyclic cohomology. The value of this pairing on the classes $[a]$ and $[\varphi]$ is equal to
$$
\langle a,\varphi\rangle=\frac 1{\sqrt{2\pi i}}\sum_{k\ge 0} (-1)^k {k!}  \varphi_{2k+1}   (a^{-1},a,\ldots,a^{-1}, a) .
$$

Now to define  the topological index of the element \eqref{eq-kth2},  it
remains to choose a cocycle over the algebra. It turns out that the desired
cocycle can be defined as a special   {\em equivariant characteristic class}
 in cyclic cohomology.

\paragraph{Equivariant characteristic classes.}

Suppose that a discrete group $G$ acts smoothly on a closed smooth  manifold
$X$. We shall also assume that  $X$ is oriented and the  action is
orientation-preserving. Let $E\in \Vect_G(X)$ be a finite-dimensional complex
$G$-bundle on $X$.  Connes defined (e.g., see \cite{Con1}) equivariant
characteristic classes of $E$ with values in cyclic cohomology
$HC^*(C^\infty(X)\rtimes_{alg} G)$ of the crossed product. However, the
formulas for these classes were quite complicated and we do not give them here.
A simple explicit formula was obtained in \cite{Gor1} for the most important
characteristic class, namely, for the equivariant Chern character
\begin{equation}
\ch_G(E) \in HC^*(C^\infty(X)\rtimes_{alg}G).
\end{equation}
More precisely, it was shown in the cited paper that the class    $\ch_G(E)$ is
represented by the collection of functionals   $\{\ch^k_G(E)\}$ defined as
\begin{multline}\label{eq-jlo1a}
\ch^k_G(E;a_0,a_1,...,a_k)=\\
=
\frac{(-1)^{(n-k)/2}}{((n+k)/2)!}
\sum_{i_0+i_1+\ldots+i_k=(n-k)/2} \int_X \tr_E
\bigl[\left( a_0 \theta^{i_0}\nabla(a_1)\theta^{i_1}\nabla(a_2)\ldots \nabla(a_k)\theta^{i_k} \right)_e\bigr]
\end{multline}
(cf.  Jaffe-Lesniewski-Osterwalder formula   \cite{JLO1}). Here
$$
\dim X=n , \quad k=n,n-2,n-4,\ldots,
$$
$\nabla_E$ is a connection in $E$ and $\theta=\nabla_E^2$ is its curvature
form,  for a noncommutative form  $\omega$ by  $\omega_e$ we denote the
coefficient of $T_e=1$, while the operator
$$
\nabla:C^\infty(X )\rtimes_{alg} G\to \Lambda^1(X,\End E)\rtimes_{alg} G
$$
is defined as
$$
\nabla (\sum_g a_g T_g)= \sum_g\bigl[
 da_g -a_g(\nabla_E- (g^{-1})^*\nabla_E) \bigr]T_g.
$$
It is proved in the cited paper that the collection of  functionals
$\{\ch^k_G(E)\}$ defines a cocycle over $C^\infty(X )\rtimes_{alg} G$, and the
class of this cocycle in cyclic cohomology does not depend on the choice  of
connection $\nabla_E$ and coincides with the equivariant Chern character
defined by Connes   \cite{Con1}.

Explicit formulas for other characteristic classes  can be obtained using
standard topological techniques   (operations in $K$-theory, see \cite{Ati2}).
For the index theorem, we need the equivariant Todd class.
\begin{proposition}[\cite{SaSt30}]
The equivariant Todd class
\begin{equation}\label{eq-td1}
 \td_G(E)\in HC^*(C^\infty(X )\rtimes_{alg} G)
\end{equation}
of a complex $G$-bundle $E$ on a smooth manifold $X$ is equal to
$$
 \td_G(E)=\ch_G(\Phi(E)),
$$
here $\Phi$ is the multiplicative operation in $K$-theory, which corresponds to the function
$\varphi(t)=t^{-1}(1+t)\ln(1+t)$.
\end{proposition}
Note that  $\Phi$ can be expressed explicitly in  terms of Grothendieck
operations. For instance, if   $\dim X\le 5$ then  (see \cite{SaSt30})
\begin{multline}\label{eq-fishka3}
\Phi(E)= 1+ \frac {E-n}2+\frac{-2(E^2-2nE+n^2)+7(E+\Lambda^2E-nE+n(n-1)/2)}{12} =\\
= \frac{3n^2-19n+24}{24}+
                           \frac{(-3n+13)}{12}E-\frac 1 6 {E\otimes E} +\frac 7{12}\Lambda^2 E
\end{multline}
where $n=\dim E$.

\paragraph{Index theorems}

\begin{theorem}[\cite{SaSt30}]
Let   $D$ be an elliptic operator associated with the action of group $\mathbb{Z}$. Then we have the index formula
\begin{equation}\label{eq-ind-forla1}
 \ind D=(2\pi i)^{-n}
\langle [\sigma(D)],\td_\mathbb{Z} (\pi^*T^*_\mathbb{C}M)\rangle, \qquad \dim M=n,
\end{equation}
where $\pi: S^*M\lra M$ is the natural projection and  the brackets
$\langle,\rangle$ denote the pairing of  $K$-theory and cyclic cohomology (see
\eqref{eq-pairing1}).
\end{theorem}

\begin{remark}
An index formula for operators on the circle associated  with an action of an
arbitrary torsion free group is announced in \cite{Per4}. The index formula in
this case has the same form as   \eqref{eq-ind-forla1}.
\end{remark}

\paragraph{Examples}

1. Suppose that the (usual)  Todd class
$\td(T^*_\mathbb{C}M)$ is trivial and the diffeomorphisms of the
$\mathbb{Z}$-action are isotopic to the identity. Then one can show that  the equivariant Todd
class is equal to the  transverse fundamental cycle (see \cite{Con8}) of $S^*M$
and the index formula
 \eqref{eq-ind-forla1} is written as:
\begin{equation}\label{eq-nice-index1}
\ind D=\frac{(n-1)!}{(2\pi i)^n(2n-1)!}\int_{S^*M}  (\sigma^{-1}d\sigma)^{2n-1}_e,\qquad \sigma=\sigma(D).
\end{equation}

2. Suppose that the group acts isometrically. Then formula   \eqref{eq-ind-forla1}
reduces to \eqref{eq-index-2} This is obvious ifwe  if we choose in the first formula invariant metric and connection on the cotangent bundle.

\section{Elliptic operators for compact Lie groups}

\subsection{Main definitions}

Let a compact Lie group $G$ act smoothly on a closed smooth manifold  $M$. An element $g\in G$
takes  a point  $x\in M$ to the point denoted by
$g(x)$. We fix a $G$-invariant metric on  $M$  and the   Haar measure on $G$.

Consider the representation  $g\mapsto T_g$ of $G$ in the space $L^2(M)$
by shift operators
$$
 T_gu(x)=u(g^{-1}(x)).
$$

\begin{definition}\label{def-nc-op2}
A  \emph{$G$-pseudodifferential operator} ($G$-$\psi$DO) is an operator
$D:L^2(M)\longrightarrow L^2(M)$ of the form
\begin{equation}\label{eq-oper-nash}
   D=1+ \int\limits_G D_gT_gdg,
\end{equation}
where $D_g$, $g\in G$ is a smooth family of pseudodifferential operators
 of order zero on $M$.
\end{definition}
Consider the   equation
\begin{equation}\label{eq-first}
 u+\int_G D_gT_gu dg=f, \quad u,f\in L^2(M).
\end{equation}
Note that if   $G$ is discrete, then we obtain the class of equations \eqref{op-shift1}.

\paragraph{Example 1. Integro-differential equations on the torus.}
On the torus $\mathbb{T}^2=\mathbb{S}^1\times \mathbb{S}^1$ with coordinates $x_1,x_2$, consider
the integro-dif\-fe\-rential equation
$$
\Delta u(x_1,x_2)+\alpha\frac{\partial^2 }{\partial x_1^2}\int_{\mathbb{S}^1}
u(x_1,y)dy=f(x_1,x_2),
$$
where $\Delta$ stands for the nonnegative Laplace operator, and $\alpha$
is a constant. Let us write this equation as
\begin{equation}\label{eq-op2z}
\Delta u+\alpha\frac{\partial^2 }{\partial x_1^2}\int_{\mathbb{S}^1} T_g dg u=f,
\end{equation}
where $T_g$ denotes the shift operator $ T_g u(x_1,x_2)=u(x_1,x_2-g), $ induced
by the action of the circle  $G=\mathbb{S}^1$ by shifts in    $x_2$. Note that
if we multiply the equation~\eqref{eq-op2z} on the left by the almost inverse
operator $\Delta^{-1}$, we obtain an equation of the type \eqref{eq-first}.

\paragraph{Example 2. Integro-differential equations on the plane.}
Consider the integro-dif\-fe\-rential equation
\begin{multline*}
\Delta u(x,y)+\\ \left(
           \alpha\frac{\partial^2 }{\partial x^2}+
           \beta\frac{\partial^2 }{\partial x\partial y}+
           \gamma \frac{\partial^2 }{\partial y^2}
         \right)\int_{\mathbb{S}^1} u(x \cos\varphi-y\sin\varphi,x\sin\varphi+y\cos\varphi) d\varphi
           =f(x,y)
\end{multline*}
on the plane $\mathbb{R}^2_{x,y}$, where  $\Delta$ is the Laplace operator, and $\alpha,\beta,\gamma$
are constants. This equation can be written as
\begin{equation}\label{eq-rota1}
\Delta u+\left(
           \alpha\frac{\partial^2 }{\partial x^2}+
           \beta\frac{\partial^2 }{\partial x\partial y}+
           \gamma \frac{\partial^2 }{\partial y^2}
         \right)\int_{\mathbb{S}^1} T_\varphi d\varphi u =f,
\end{equation}
where the shift operator $T_\varphi$ is induced by the action of the circle  $G=\mathbb{S}^1$ by rotations
$$
(x,y)\longmapsto (x\cos\varphi+y\sin\varphi,-x\sin\varphi+y\cos\varphi)
$$
around the origin. If we multiply the equation \eqref{eq-rota1} on the left by the almost inverse
operator $\Delta^{-1}$, we obtain an equation similar to \eqref{eq-first}.

\subsection{Pseudodifferential uniformization}

Here we formulate an approach, called {\em pseudodifferential uniformization}, which enables one to
reduce a $G$-pseudodifferential operator
\begin{equation}\label{eq-gg1}
D=1+\int_G D_gT_g dg:L^2(M)\longrightarrow L^2(M).
\end{equation}
to a pseudodifferential operator and then apply
the methods of the theory of pseudodifferential operators.

\paragraph{1. Reduction to a $\psi$DO}
This reduction is constructed as follows.

\begin{enumerate}
    \item[$\bullet$] The operator $D$ is represented as an operator on the quotient  $M/G$
          (the space of orbits).
    \item[$\bullet$] If the action of $G$ on $M$ has no fixed points, then
       $M/G$ is a smooth manifold; moreover,   $D$ can be treated as a $\psi$DO on
    $M/G$ with operator-valued symbol in the sense of Luke \cite{Luk1}  (explanation: this follows from the fact
    that the operator $T_g$ acts only along the fibers of the infinite-dimensional bundle over
    $M/G$, but not along the base).
    \item[$\bullet$] If the  fixed point set  is nonempty,
     then     $M/G$ has singularities; to construct a $\psi$DO  in the case,   we do the following.
    \item[$\bullet$] We lift  $D$ from $M$ to the product $M\times G$
    endowed with the diagonal action of   $G$:
    \begin{equation}\label{diag-action}
     (x,h) \longmapsto (g(x),gh).
    \end{equation}
    \item[$\bullet$] The action \eqref{diag-action} is fixed point free. Hence, the obtained
    $G$-pseudodifferential operator on   $M\times G$, which we denote by  $\widetilde{D}$,
can be represented (see above) as a $\psi$DO on the smooth orbit space  $(M\times G)/G\simeq M$.
\end{enumerate}

These steps give the commutative diagram
\begin{equation}\label{eq-reduction1}
\xymatrix{
  L^2(M)\ar[r]^D\ar[d]_{\pi^*} & L^2(M) \ar[d]^{\pi^*}\\
  L^2(M\times G) \ar[r]^{\widetilde{D}} \ar[d]^{\simeq} & L^2(M\times G) \ar[d]_{\simeq} \\
  L^2(M,L^2(G)) \ar[r]^{\mathcal{D}} & L^2(M,L^2(G)),
 }
\end{equation}
where $\pi^*$ is the induced mapping for the projection $\pi:M\times G\to M$, while
$\mathcal{D}$ is the pseudodifferential operator on $M$.

\begin{remark}\label{rek1}{
The fact that   $\mathcal{D}$ is a $\psi$DO is clear for geometric reasons.
Indeed,      $\widetilde{D}$ has shifts along the orbits of the diagonal action
of   $G$ (see~Figure~1, left
). Clearly, these orbits can be transformed into  vertical orbits (see~Figure~1, right)
by a  change of variables on $M\times G$.
\begin{figure}\label{ris-diag}
 \begin{center}
 \includegraphics[width=3cm
 ]{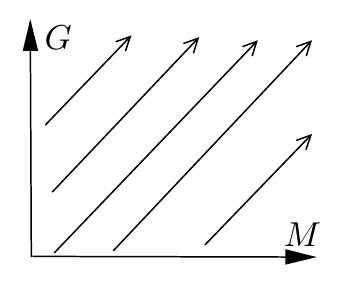}\hspace{2cm}
\includegraphics[width=3cm
 ]{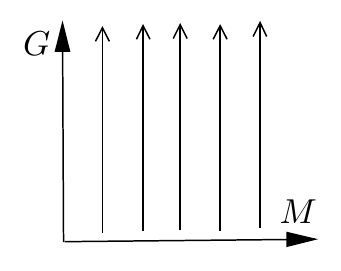}
\caption{The orbits of the diagonal and vertical actions on the product $M\times G$.}
  \end{center}
\end{figure}
The shift operator along the vertical orbits is a $\psi$DO on   $M$.}
\end{remark}

\paragraph{2. Restriction of $\psi$DO to the subspace of invariant sections.}

The mapping $\pi^*$ in \eqref{eq-reduction1} is   a monomorphism.
Its range is the space of  $G$-invariant sections. Hence,
\eqref{eq-reduction1} gives a commutative diagram of the form
$$
          \xymatrix{
           L^2(M)\ar[r]^D \ar[d]^\simeq
                                       & L^2(M) \ar[d]_\simeq
                                                             \\
           L^2(M,L^2(G))^G \ar[r]^{\mathcal{D}^G} & L^2(M,L^2(G))^G,
          }
$$
where $\mathcal{D}^G$ stands for the restriction of   $\mathcal{D}$ to the subspace of
invariant sections, which we denote by   $L^2(M,L^2(G))^G$.

\paragraph{3. Transverse ellipticity.}

It remains to give conditions,   which imply that the restriction
$$
 L^2(M,L^2(G))^G \stackrel{\mathcal{D}^G}\longrightarrow   L^2(M,L^2(G))^G
$$
of   $\mathcal{D}$ to the subspace of  $G$-invariant sections is Fredholm. Let
us note that  invariant sections are constant along the  orbits of the group
action. Hence, it suffices to impose the condition, which guarantees the
Fredholm property,
    only along the transverse directions to the orbit (see~Figure~2
).
\begin{figure}\label{ris-transv}
  \begin{center}
\includegraphics[width=6cm
 ]{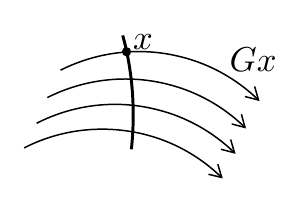}\caption{Transversal to the orbit.}
  \end{center}
\end{figure}

\begin{definition}[\cite{Ati8,Sin3}]\label{defo1}
A pseudodifferential operator $ {\mathcal{D}}$ is    {\em transversally elliptic},
if its symbol $\sigma( {\mathcal{D}})(x,\xi)$ is invertible for all
$(x,\xi)\in T^*_GM \setminus  0$, where
$$
T^*_GM=\{(x,\xi)\in  T^*M \;|\; \text{ covector } \xi \text{ is orthogonal to the orbit }Gx\}.
$$
stands for the {\em transverse cotangent bundle}.
\end{definition}

\begin{theorem}[\cite{SaSt22,Ster20}]
A transversally elliptic  operator $\mathcal{D}$  restricts to a Fredholm operator on the subspace of  $G$-invariant sections
$$
\mathcal{D}^G:L^2(M,L^2(G))^G\longrightarrow L^2(M,L^2(G))^G.
$$
\end{theorem}

Let us summarize the above discussion.

\begin{enumerate}
    \item To a $G$-pseudodifferential operator $D$ we assigned a pseudodifferential operator
     $\mathcal{D}$ such that there is an isomorphism
    \begin{equation}\label{eq-isom6}
    D\simeq \mathcal{D}^G,
    \end{equation}
    where $ \mathcal{D}^G$ is the restriction of $\mathcal{D}$ to the subspace of invariant sections.
    \item If   $\mathcal{D}$ is transversally elliptic, then its restriction
      $\mathcal{D}^G$ is Fredholm. Hence, by virtue of the isomorphism  \eqref{eq-isom6}
    the original   operator $D$ is also Fredholm.
\end{enumerate}

Since we have  isomorphism \eqref{eq-isom6}, we obtain:
$$
\ind D=\ind \mathcal{D}^G.
$$
Using this equation and the  theory of transversally  elliptic
pseudodifferential operators~\cite{Ati8,Kaw1,Verg1}), an index theorem for the 
$G$-operator $D$ was obtained in \cite{SaSt22,Sav11}.  We only mention here
that the main ingredients of the index formula are: 1) the
definition of the symbol of $D$ as an element of the crossed product of the
algebra of functions on the transverse cotangent bundle by the group $G$; 2)
a Chern character mapping on the $K$-theory of this algebra ranging in the
basic cohomology of  fixed point sets of the group action.


\begin{thebibliography}{10}

\bibitem{AnLe2}
A.~Antonevich, M.~Belousov, and A.~Lebedev.
\newblock {\em Functional differential equations. {II. $C^*$}-applications.
  Parts 1,~2}.
\newblock Number 94, 95 in Pitman Monographs and Surveys in Pure and Applied
  Mathematics. Longman, Harlow, 1998.

\bibitem{AnLe1}
A.~Antonevich and A.~Lebedev.
\newblock {\em Functional-Differential Equations. {I. $C\sp *$}-Theory}.
\newblock Number~70 in Pitman Monographs and Surveys in Pure and Applied
  Mathematics. Longman, Harlow, 1994.

\bibitem{Ant4}
A.~B. Antonevich.
\newblock Elliptic pseudodifferential operators with a finite group of shifts.
\newblock {\em Math. USSR-Izv.}, 7:661--674, 1973.

\bibitem{Ant2}
A.~B. Antonevich.
\newblock {\em Linear functional equations. Operator apporach}.
\newblock Universitetskoje, Minsk, 1988.

\bibitem{Ant1}
A.~B. Antonevich.
\newblock Strongly nonlocal boundary value problems for elliptic equations.
\newblock {\em Izv. Akad. Nauk SSSR Ser. Mat.}, 53(1):3--24, 1989.

\bibitem{AnBr1}
A.~B. Antonevich and V.V. Brenner.
\newblock On the symbol of a pseudodifferential operator with locally
  independent shifts.
\newblock {\em Dokl. Akad. Nauk BSSR}, 24(10):884--887, 1980.

\bibitem{AnLe3}
A.~B. Antonevich and A.~V. Lebedev.
\newblock Functional equations and functional operator equations. {A} {$C\sp
  \ast$}-algebraic approach.
\newblock In {\em Proceedings of the St. Petersburg Mathematical Society, Vol.
  VI}, volume 199 of {\em Amer. Math. Soc. Transl. Ser. 2}, pages 25--116,
  Providence, RI, 2000. Amer. Math. Soc.

\bibitem{Ati8}
M.~F. Atiyah.
\newblock {\em Elliptic operators and compact groups}.
\newblock Lecture Notes in Mathematics, Vol. 401. Springer-Verlag, Berlin,
  1974.

\bibitem{Ati2}
M.~F. Atiyah.
\newblock {\em K-Theory}.
\newblock The Advanced Book Program. Addison--Wesley, Inc., second edition,
  1989.

\bibitem{AtSe2}
M.~F. Atiyah and G.~B. Segal.
\newblock The index of elliptic operators {II}.
\newblock {\em Ann. Math.}, 87:531--545, 1968.

\bibitem{AtSi0}
M.~F. Atiyah and I.~M. Singer.
\newblock The index of elliptic operators on compact manifolds.
\newblock {\em Bull. Amer. Math. Soc.}, 69:422--433, 1963.

\bibitem{AtSi3}
M.~F. Atiyah and I.~M. Singer.
\newblock The index of elliptic operators {III}.
\newblock {\em Ann. Math.}, 87:546--604, 1968.

\bibitem{Bab1}
Ch. Babbage.
\newblock An assay towards the calculus of functions. part {II}.
\newblock {\em Philos. Trans. of the {R}oyal Society}, 106:179--256, 1816.

\bibitem{BaCo2}
P.~Baum and A.~Connes.
\newblock Chern character for discrete groups.
\newblock In {\em A f\^ete of topology}, pages 163--232. Academic Press,
  Boston, MA, 1988.

\bibitem{BiSo3}
A.V. Bitsadze and A.A. Samarskii.
\newblock {On some simple generalizations of linear elliptic boundary
  problems}.
\newblock {\em Sov. Math., Dokl.}, 10:398--400, 1969.

\bibitem{Bla1}
B.~Blackadar.
\newblock {\em $K$-Theory for Operator Algebras}.
\newblock Number~5 in Mathematical Sciences Research Institute Publications.
  Cambridge University Press, 1998.
\newblock Second edition.

\bibitem{Carl1}
T.~Carleman.
\newblock {Sur la th\'eorie des \'equations int\'egrales et ses applications.}
\newblock pages 138--151, 1932.
\newblock {Verh. Internat. Math.-Kongr. Zurich. 1.}

\bibitem{CoFl1}
P.~E. Conner and E.~E. Floyd.
\newblock {\em Differentiable periodic maps}.
\newblock Academic Press, New York, 1964.

\bibitem{Con4}
A.~Connes.
\newblock {$C\sp{\ast} $} alg\`ebres et g\'eom\'etrie diff\'erentielle.
\newblock {\em C. R. Acad. Sci. Paris S\'er. A-B}, 290(13):A599--A604, 1980.

\bibitem{Con6}
A.~Connes.
\newblock Noncommutative differential geometry.
\newblock {\em Inst. Hautes \'Etudes Sci. Publ. Math.}, (62):257--360, 1985.

\bibitem{Con8}
A.~Connes.
\newblock Cyclic cohomology and the transverse fundamental class of a
  foliation.
\newblock In {\em Geometric methods in operator algebras}, volume 123 of {\em
  Pitman Res. Notes in Math.}, pages 52--144. Longman, Harlow, 1986.

\bibitem{Con1}
A.~Connes.
\newblock {\em Noncommutative geometry}.
\newblock Academic Press Inc., San Diego, CA, 1994.

\bibitem{CoDu1}
A.~Connes and M.~Dubois-Violette.
\newblock Noncommutative finite-dimensional manifolds. {I}. {S}pherical
  manifolds and related examples.
\newblock {\em Comm. Math. Phys.}, 230(3):539--579, 2002.

\bibitem{CoLa2}
A.~Connes and G.~Landi.
\newblock Noncommutative manifolds, the instanton algebra and isospectral
  deformations.
\newblock {\em Comm. Math. Phys.}, 221(1):141--159, 2001.

\bibitem{CoMo2}
A.~Connes and H.~Moscovici.
\newblock Type {III} and spectral triples.
\newblock In {\em Traces in number theory, geometry and quantum fields},
  Aspects Math., E38, pages 57--71. Friedr. Vieweg, Wiesbaden, 2008.

\bibitem{Gel1}
I.~M. Gelfand.
\newblock On elliptic equations.
\newblock {\em Russian Math. Surveys}, 15(3):113--127, 1960.

\bibitem{Gor1}
A.~Gorokhovsky.
\newblock Characters of cycles, equivariant characteristic classes and
  {F}redholm modules.
\newblock {\em Comm. Math. Phys.}, 208(1):1--23, 1999.

\bibitem{Gro1}
M.~Gromov.
\newblock Groups of polynomial growth and expanding maps.
\newblock {\em Inst. Hautes \'Etudes Sci. Publ. Math.}, (53):53--73, 1981.

\bibitem{JLO1}
A.~Jaffe, A.~Lesniewski, and K.~Osterwalder.
\newblock Quantum {$K$}-theory. {I}. {T}he {C}hern character.
\newblock {\em Comm. Math. Phys.}, 118(1):1--14, 1988.

\bibitem{Kaw1}
T.~Kawasaki.
\newblock The index of elliptic operators over {$V$}-manifolds.
\newblock {\em Nagoya Math. J.}, 84:135--157, 1981.

\bibitem{Kord4}
Yu.~A. Kordyukov.
\newblock Transversally elliptic operators on ${G}$-manifolds of bounded
  geometry.
\newblock {\em Russian J. Math. Phys.}, 2(2):175--198, 1994.

\bibitem{Kord5}
Yu.~A. Kordyukov.
\newblock Transversally elliptic operators on ${G}$-manifolds of bounded
  geometry. {II}.
\newblock {\em Russian J. Math. Phys.}, 3(1):41--64, 1995.

\bibitem{Kord2}
Yu.~A. Kordyukov.
\newblock Index theory and non-commutative geometry on foliated manifolds.
\newblock {\em Russ. Math. Surv.}, 64(2):273--391, 2009.

\bibitem{LaSu1}
G.~Landi and W.~van Suijlekom.
\newblock Principal fibrations from noncommutative spheres.
\newblock {\em Comm. Math. Phys.}, 260(1):203--225, 2005.

\bibitem{Luk1}
G.~Luke.
\newblock Pseudodifferential operators on {H}ilbert bundles.
\newblock {\em J. Diff. Equations}, 12:566--589, 1972.

\bibitem{MiFo1}
A.~S. Mishchenko and A.~T. Fomenko.
\newblock The index of elliptic operators over {$C\sp{\ast} $}-algebras.
\newblock {\em Izv. Akad. Nauk SSSR Ser. Mat.}, 43(4):831--859, 967, 1979.

\bibitem{Mos2}
H.~Moscovici.
\newblock Local index formula and twisted spectral triples.
\newblock In {\em Quanta of maths}, volume~11 of {\em Clay Math. Proc.}, pages
  465--500. Amer. Math. Soc., Providence, RI, 2010.

\bibitem{NaSaSt17}
V.~E. Nazaikinskii, A.~Yu. Savin, and B.~Yu. Sternin.
\newblock {\em Elliptic theory and noncommutative geometry}, volume 183 of {\em
  Operator Theory: Advances and Applications}.
\newblock Birkh\"auser Verlag, Basel, 2008.

\bibitem{Park1}
E.~Park.
\newblock Index theory of {T}oeplitz operators associated to transformation
  group {$C\sp *$}-algebras.
\newblock {\em Pacific J. Math.}, 223(1):159--165, 2006.

\bibitem{Pat2}
A.~L.~T. Paterson.
\newblock {\em Amenability}, volume~29 of {\em Mathematical Surveys and
  Monographs}.
\newblock American Mathematical Society, Providence, RI, 1988.

\bibitem{Ped1}
G.K. Pedersen.
\newblock {\em {$C^*$}-Algebras and Their Automorphism Groups}, volume~14 of
  {\em London Mathematical Society Monographs}.
\newblock Academic Press, London--New York, 1979.

\bibitem{Per3}
D.~Perrot.
\newblock A {R}iemann-{R}och theorem for one-dimensional complex groupoids.
\newblock {\em Comm. Math. Phys.}, 218(2):373--391, 2001.

\bibitem{Per2}
D.~Perrot.
\newblock Localization over complex-analytic groupoids and conformal
  renormalization.
\newblock {\em J. Noncommut. Geom.}, 3(2):289--325, 2009.

\bibitem{Per4}
D.~Perrot.
\newblock On the {R}adul cocycle.
\newblock {\em Oberwolfach reports}, pages 53--55, 2011.
\newblock DOI: 10.4171/OWR/2011/45.

\bibitem{Ross1}
L.E. Rossovskii.
\newblock Boundary value problems for elliptic functional-differential
  equations with dilatation and contraction of the arguments.
\newblock {\em Trans. Moscow Math. Soc.}, pages 185--212, 2001.

\bibitem{SaSchSt3}
A.~Savin, E.~Schrohe, and B.~Sternin.
\newblock On the index formula for an isometric diffeomorphism.
\newblock arXiv:1112.5515, 2011.

\bibitem{SaSchSt2}
A.~Savin, E.~Schrohe, and B.~Sternin.
\newblock Uniformization and an index theorem for elliptic operators associated
  with diffeomorphisms of a manifold.
\newblock arXiv:1111.1525, 2011.

\bibitem{SaSt10}
A.~Savin and B.~Sternin.
\newblock Index defects in the theory of nonlocal boundary value problems and
  the $\eta$-invariant.
\newblock {\em Sbornik: Mathematics}, 195(9):1321--1358, 2004.
\newblock arXiv: math/0108107.

\bibitem{SaSt30}
A.~Savin and B.~Sternin.
\newblock Index of elliptic operators for a diffeomorphism.
\newblock arxiv:1106.4195, 2011.

\bibitem{Sav11}
A.~Yu. Savin.
\newblock On the index of nonlocal elliptic operators for compact {L}ie groups.
\newblock {\em Cent. Eur. J. Math.}, 9(4):833--850, 2011.

\bibitem{Sav9}
A.~Yu. Savin.
\newblock On the index of nonlocal operators associated with a nonisometric
  diffeomorphism.
\newblock {\em Mathematical Notes}, 90(5):701--714, 2011.

\bibitem{Sav12}
A.~Yu. Savin.
\newblock On the symbol of nonlocal operators in {S}obolev spaces.
\newblock {\em Differential Equations}, 47(6):897--900, 2011.

\bibitem{SaSt20}
A.~Yu. Savin and B.~Yu. Sternin.
\newblock Index of nonlocal elliptic operators over ${C}^*$-algebras.
\newblock {\em Dokl. Math.}, 79(3):369--372, 2009.

\bibitem{SaSt23}
A.~Yu. Savin and B.~Yu. Sternin.
\newblock Noncommutative elliptic theory. {E}xamples.
\newblock {\em Proceedings of the {S}teklov {I}nstitute of {M}athematics},
  271:193--211, 2010.

\bibitem{SaSt22}
A.~Yu. Savin and B.~Yu. Sternin.
\newblock Nonlocal elliptic operators for compact {L}ie groups.
\newblock {\em Dokl. Math.}, 81(2):258--261, 2010.

\bibitem{Sav10}
A.Yu. Savin.
\newblock On the index of elliptic operators associated with a diffeomorphism
  of a manifold.
\newblock {\em Doklady Mathematics}, 82(3):884--886, 2010.

\bibitem{SaSt26}
A.Yu. Savin and B.~Yu. Sternin.
\newblock On the index of noncommutative elliptic operators over
  ${C}^*$-algebras.
\newblock {\em Sbornik. Mathematics}, 201(3):377--417, 2010.

\bibitem{SaSt25}
A.Yu. Savin and B.~Yu. Sternin.
\newblock Nonlocal elliptic operators for the group of dilations.
\newblock {\em Sbornik. Mathematics}, 202(10):1505--1536, 2011.

\bibitem{SaSchSt1}
A.Yu. Savin, B.~Yu. Sternin, and E.~Schrohe.
\newblock Index problem for elliptic operators associated with a diffeomorphism
  of a manifold and uniformization.
\newblock {\em Dokl.Math.}, 84(3):846--849, 2011.

\bibitem{Schwe1}
L.~B. Schweitzer.
\newblock Spectral invariance of dense subalgebras of operator algebras.
\newblock {\em Internat. J. Math.}, 4(2):289--317, 1993.

\bibitem{Sin3}
I.~M. Singer.
\newblock Recent applications of index theory for elliptic operators.
\newblock In {\em Partial differential equations ({P}roc. {S}ympos. {P}ure
  {M}ath., {V}ol. {XXIII}, {U}niv. {C}alifornia, {B}erkeley, {C}alif., 1971)},
  pages 11--31. Amer. Math. Soc., Providence, R.I., 1973.

\bibitem{Sku1}
A.L. Skubachevskii.
\newblock {\em Elliptic functional differential equations and applications}.
\newblock Birkh{\"a}user, Basel-Boston-Berlin, 1997.

\bibitem{Sku2}
A.L. Skubachevskii.
\newblock Nonclassical boundary-value problems. {I}.
\newblock {\em Journal of Mathematical Sciences}, 155(2):199--334, 2008.

\bibitem{Sku3}
A.L. Skubachevskii.
\newblock Nonclassical boundary-value problems. {II}.
\newblock {\em Journal of Mathematical Sciences}, 166(4):377--561, 2010.

\bibitem{Slo1}
J.~Slominska.
\newblock {On the equivariant {C}hern homomorphism.}
\newblock {\em Bull. Acad. Pol. Sci., Ser. Sci. Math. Astron. Phys.},
  24:909--913, 1976.

\bibitem{Ster20}
B.~Yu. Sternin.
\newblock On a class of nonlocal elliptic operators for compact {L}ie groups.
  {U}niformization and finiteness theorem.
\newblock {\em Cent. Eur. J. Math.}, 9(4):814--832, 2011.

\bibitem{Verg1}
M.~Vergne.
\newblock Equivariant index formulas for orbifolds.
\newblock {\em Duke Math. J.}, 82(3):637--652, 1996.

\bibitem{Wil1}
D.~P. Williams.
\newblock {\em Crossed products of {$C^*$}-algebras}, volume 134 of {\em
  Mathematical Surveys and Monographs}.
\newblock American Mathematical Society, Providence, RI, 2007.

\bibitem{Zell1}
G.~Zeller-Meier.
\newblock Produits crois\'es d'une {$C^{\ast} $}-alg\`ebre par un groupe
  d'automorphismes.
\newblock {\em J. Math. Pures Appl. (9)}, 47:101--239, 1968.

\end{thebibliography}

\end{document}